\documentclass[]{aspm}

\articleinfo{}{}{}

\usepackage{verbatim}
\usepackage{amssymb}
\usepackage{amsbsy}
\usepackage{amscd}
\usepackage{amsmath}
\usepackage{amsthm}
\usepackage[mathscr]{eucal}

\usepackage[all]{xy}
\usepackage{graphicx}
\usepackage{tikz,tikz-cd}
\usepackage{multirow}

\newtheorem{theorem}{Theorem}\numberwithin{theorem}{section}
\newtheorem{lemma}[theorem]{Lemma}
\newtheorem{proposition}[theorem]{Proposition}

\theoremstyle{definition} % style definition
\newtheorem{definition}[theorem]{Definition}
\newtheorem{remark}[theorem]{Remark}
\newtheorem{example}[theorem]{Example}

\DeclareMathOperator{\cA}{\mathcal{A}}

\DeclareMathOperator{\cF}{\mathcal{F}}

\DeclareMathOperator{\cP}{\mathcal{P}}

\DeclareMathOperator{\bQ}{\mathbb{Q}}
\DeclareMathOperator{\bR}{\mathbb{R}}
\DeclareMathOperator{\bZ}{\mathbb{Z}}

\DeclareMathOperator{\fD}{\mathfrak{D}}

\DeclareMathOperator{\be}{\mathbf{e}}
\DeclareMathOperator{\bx}{\mathbf{x}}

\renewcommand{\mod}{{\rm mod}}

\DeclareMathOperator{\diag}{{\rm diag}}

\DeclareMathOperator{\add}{{\rm add}}

\DeclareMathOperator{\proj}{{\rm proj}}
\DeclareMathOperator{\rad}{{\rm rad}}
\DeclareMathOperator{\Hom}{{\rm Hom}}

\DeclareMathOperator{\twosilt}{{\rm 2\mathchar`-silt}}
\DeclareMathOperator{\ipsilt}{{\rm 2\mathchar`-ipsilt}}

\title[Acyclic cluster algebras with dense $g$-vector fans]{Acyclic cluster algebras with dense $g$-vector fans}

\author[T. Yurikusa]{Toshiya Yurikusa}

\address{Mathematical Institute, Tohoku University, Aoba-ku, Sendai, 980-8578, Japan}

\email{toshiya.yurikusa.d8@tohoku.ac.jp}

\rcvdate{}
\rvsdate{}



\thanks{The author was supported by Japan Society for the Promotion of Science KAKENHI Grant Numbers JP20J00410, JP21K13761.}

\begin{document}

\begin{abstract}
The $g$-vector fans play an important role in studying cluster algebras and silting theory. We survey cluster algebras with dense $g$-vector fans and show that a connected acyclic cluster algebra has a dense $g$-vector fan if and only if it is either finite type or affine type. As an application, we classify finite dimensional hereditary algebras with dense $g$-vector fans.
\end{abstract}

\maketitle

%%%
%%%
%%%
\section{Introduction}

Fomin and Zelevinsky \cite{FZ02} introduced cluster algebras, which are commutative algebras with generators called cluster variables. The clusters are certain tuples of cluster variables. Their original motivation was to study total positivity of semisimple Lie groups and canonical bases of quantum groups. In recent years, cluster algebras have appeared in various areas of mathematics, for example, representation theory of quivers, Poisson geometry, and integrable systems.

Let $B$ be a connected skew-symmetrizable $n \times n$ matrix and $\cA(B)$ the associated cluster algebra with principal coefficients (see Subsection \ref{subsec:cluster algebras}). Each cluster variable of $\cA(B)$ has a numerical invariant, called the $g$-vector, as an integer vector in $\bR^n$. The $g$-vectors in $\cA(B)$ form a fan $\cF(B)$, called the $g$-vector fan of $\cA(B)$, whose cones are spanned by the $g$-vectors of some cluster variables belonging to a cluster (see Subsection \ref{subsec:g-vector fans}). In particular, $\cF(B)$ is a set of cones in $\bR^n$. The $g$-vector fans appear in various subjects (or, as their subfans), for example, the tropical cluster $\mathcal{X}$-varieties \cite{FG}, the cluster or stability scattering diagrams \cite{B,GHKK}, and the normal fans of some polytopes known as the generalized associahedron \cite{AHBHY,BMDMTY,CFZ,HPS,PPPP}.

It is well-known that $B$ is finite type if and only if $\cF(B)$ is complete (Theorem \ref{thm:finite type complete}). In this paper, we are interested in cluster algebras with dense $g$-vector fans, that is, the union of its cones is dense in $\bR^n$. When $B$ is skew-symmetric, such cluster algebras were classified in \cite{PY} except for type $X_6$ (Theorem \ref{thm:classify g-dense cluster}). In non-skew-symmetric cases, we don't know whether the $g$-vector fan is dense or not except for rank $2$, finite type or affine type (see Subsection \ref{subsec:dense g-vector fans}).

The matrix $B$ is called mutation-finite if there are finitely many matrices which are mutation equivalent to $B$. Mutation-finite skew-symmetrizable matrices were classified in \cite{FeST12a,FeST12b}. The class plays an important role in studying cluster algebras with dense $g$-vector fans. In fact, we show that if $\cF(B)$ is dense, then $B$ is mutation-finite (Theorem \ref{thm:mutation-infinite g-dense}).

In this paper, we review the classifications of mutation-finite skew-symmetrizable matrices and study cluster algebras with dense $g$-vector fans. In particular, we restrict them to acyclic cases. Seven \cite{Se11} showed that an acyclic connected skew-symmetrizable matrix is mutation-finite if and only if it is rank $2$, finite type or affine type (Theorem \ref{thm:mutation-finite acyclic}). Applying them for cluster algebras, we give a complete classification of acyclic cluster algebras with dense $g$-vector fans.

%%%
\begin{theorem}\label{thm:g-dense acyclic}
A connected acyclic cluster algebra has a dense $g$-vector fan if and only if it is either finite type or affine type.
\end{theorem}

As an application, we consider $g$-vector fans in $2$-term silting theory (see Subsection \ref{subsec:g-vector fans on 2-term}). They play an important role in the study of stability conditions \cite{Kin} and their wall-chamber structures (see e.g. \cite{AHIKM,A,BST,DIJ,T,Y18}). Applying Theorem \ref{thm:g-dense acyclic}, we prove the following.

\begin{theorem}\label{thm:g-tame hereditary}
A finite dimensional, connected, hereditary algebra over an algebraically closed field $K$ has a dense $g$-vector fan if and only if it is Morita equivalent to the path algebra of a quiver of Dynkin or extended Dynkin type.
\end{theorem}

This paper is organized as follows. In Section \ref{sec:cluster algebras}, we recall the notions of cluster algebras and their $g$-vector fans. In particular, we survey two classes of mutation-finite cluster algebras and ones with dense $g$-vector fans. We give the non-denseness of $g$-vector fans for infinite mutation type (Theorem \ref{thm:mutation-infinite g-dense}) in Subsection \ref{subsec:infinite mutation type} and prove Theorem \ref{thm:g-dense acyclic} in Subsection \ref{subsec:acyclic}. In Subection \ref{subsec:future problems}, we summarize what we already know for non-acyclic cases. In Section \ref{sec:rep theory}, we recall $2$-term silting theory and $g$-vector fans on that. We get Theorem \ref{thm:g-tame hereditary} via an additive categorification of cluster algebras in Subsection \ref{subsec:categorification}.

%%% Section : Cluster algebras with dense $g$-vector fans
%%%
%%%
\section{Cluster algebras with dense $g$-vector fans}\label{sec:cluster algebras}

\subsection{Mutation of matrices}

Let $n$ be a positive integer. We say that an $n \times n$ integer matrix $B=(b_{ij})$ is \emph{skew-symmetrizable} if there is a positive integer diagonal matrix $D=\diag(d_1,\ldots,d_n)$ such that $DB$ is skew-symmetric, that is, $d_i b_{ij}= - d_j b_{ji}$. We consider an operation that gives new matrices from a given matrix, which plays an important role in defining cluster algebras.

\begin{definition}\label{def:matrix mutation}
Let $m \ge n$ be positive integers and $B=(b_{ij})$ an $m\times n$ integer matrix whose upper part $(b_{ij})_{1 \le i, j \le n}$ is skew-symmetrizable. The \emph{mutation of $B$ at $k$} $(1 \le k \le n)$ is the matrix $\mu_k B=(b'_{ij})$ given by
\[
b'_{ij} =  \left\{\begin{array}{ll}
 -b_{ij} & \mbox{if} \ \ i=k \ \ \mbox{or} \ \ j=k,\\
 b_{ij}+b_{ik}[b_{kj}]_+ + [-b_{ik}]_+b_{kj} & \mbox{otherwise},
\end{array} \right.
\]
where $[a]_+:={\rm max}(a,0)$.
\end{definition}

It is easy to check that the upper part of $\mu_k B$ is still skew-symmetrizable. Moreover, $\mu_k$ is an involution, that is, $\mu_k \mu_k B = B$.

\begin{example}\label{exam:rank2 mutations}
Any skew-symmetrizable $2 \times 2$ nonzero matrix is given by
\[
B_{b,c}:=\begin{bmatrix}0 & c \\ -b & 0\end{bmatrix},
\]
where $b$ and $c$ are integers and $bc > 0$. Then $\mu_i(B_{b,c})=B_{-b,-c}$ for $i \in \{1,2\}$. Moreover, for the $4 \times 2$ matrix $\hat{B}_{2,1}$ (see Subsection \ref{subsec:cluster algebras}), we have the sequence of mutations as follows:
\[
\begin{tikzpicture}[baseline=-1mm]
\node(0)at(0,0){$\begin{bmatrix}0 & 1 \\ -2 & 0 \\ 1 & 0 \\ 0 & 1 \end{bmatrix}$};
\node(1)at(2.5,0){$\begin{bmatrix}0 & -1 \\ 2 & 0 \\ -1 & 1 \\ 0 & 1 \end{bmatrix}$};
\node(2)at(5,0){$\begin{bmatrix}0 & 1 \\ -2 & 0 \\ 1 & -1 \\ 2 & -1 \end{bmatrix}$};
\node(3)at(0,-2){$\begin{bmatrix}0 & -1 \\ 2 & 0 \\ 1 & 0 \\ 0 & -1 \end{bmatrix}$};
\node(4)at(2.5,-2){$\begin{bmatrix}0 & 1 \\ -2 & 0 \\ -1 & 0 \\ 0 & -1 \end{bmatrix}$};
\node(5)at(5,-2){$\begin{bmatrix}0 & -1 \\ 2 & 0 \\ -1 & 0 \\ -2 & 1 \end{bmatrix}$};
\draw[<->](0)--node[above]{$\mu_1$}(1); \draw[<->](1)--node[above]{$\mu_2$}(2);
\draw[<->](3)--node[above]{$\mu_1$}(4); \draw[<->](4)--node[above]{$\mu_2$}(5);
\draw[<->](0)..controls(-2,0)and(-2,-2)..node[left]{$\mu_2$}(3);
\draw[<->](2)..controls(7,0)and(7,-2)..node[right]{$\mu_1$}(5);
\end{tikzpicture}
\]
\end{example}

%%% Finite mutation type
%%%
\subsection{Finite mutation type}\label{subsec:finite mutation type}

Let $B=(b_{ij})_{1 \le i,j \le n}$ be a skew-symmetrizable $n \times n$ matrix.

\begin{definition}
We say that $B$ is
\begin{itemize}
\item \emph{mutation equivalent to $B'$} if it is obtained from $B'$ by a sequence of mutations;
\item \emph{mutation-finite} if there are only finitely many matrices which are mutation equivalent to $B$;
\item \emph{mutation-infinite} if it is not mutation-finite.
\end{itemize}
\end{definition}

By Example \ref{exam:rank2 mutations}, $B_{b,c}$ is mutation equivalent to $B_{b',c'}$ if and only if $(b',c')=(\pm b,\pm c)$. Thus any skew-symmetrizable $2 \times 2$ matrix  is mutation-finite.

For convenience, we prepare graphs representing skew-symmetrizable matrices \cite[Definition 7.3]{FZ03}.

\begin{definition}
The \emph{diagram} of $B$ is a weighted directed graph $\Gamma(B)$ with vertices $1,\ldots,n$ such that there is an arrow with weight $|b_{ij}b_{ji}|$ from $j$ to $i$ if and only if $b_{ij}>0$.
\end{definition}

\begin{example}
For positive integers $b$ and $c$, we have the diagram
\[
\Gamma(B_{b,c})=\left[ 1 \xleftarrow{bc} 2 \right].
\]
\end{example}

Remark that distinct skew-symmetrizable matrices may have the same diagram. When $B$ is skew-symmetric, its diagram $\Gamma(B)$ is identified with a \emph{quiver}, which is a directed graph allowing multiple arrows, with vertices $1,\ldots,n$ such that there are $[b_{ij}]_+$ arrows from $j$ to $i$. This naturally restricts to a bijection between skew-symmetric matrices and quivers without loops nor oriented cycles of length two.

\emph{Dynkin} (resp., \emph{extended Dynkin}) \emph{diagrams} are Dynkin (resp., extended Dynkin) graphs as in Figure \ref{fig:(extended) Dynkin graphs} with arbitrary orientations except for orientations of $A_n^{(1)}$ forming an oriented cycle. Notice that $\Gamma(B_{b,c})$ is a Dynkin (resp., extended Dynkin) diagram if and only if $bc<4$ (resp., $bc=4$).

\begin{definition}
We say that $B$ is
\begin{itemize}
\item \emph{connected} if $\Gamma(B)$ is connected;
\item \emph{acyclic} if $\Gamma(B)$ has no oriented cycles;
\item \emph{type $\Gamma$} if it is mutation equivalent to a skew-symmetrizable matrix whose diagram (or quiver) is $\Gamma$.
\item \emph{finite type} if it is type a Dynkin diagram;
\item \emph{affine type} if it is type an extended Dynkin diagram.
\end{itemize}
\end{definition}

%%%
\begin{figure}[htp]
\begin{tabular}{c|c|c|c}
\multicolumn{2}{c|}{Dynkin graphs} & \multicolumn{2}{c}{extended Dynkin graphs}
\begin{tikzpicture}[baseline=-1mm]\node at(0,0.3){};\node at(0,-0.3){};\end{tikzpicture}
\\\hline
\multirow{4.5}{*}{$A_n$} &
\multirow{4.5}{*}{\begin{tikzpicture}[baseline=0mm,scale=0.7]
\node at(0,0){$\bullet$}; \node at(1,0){$\bullet$}; \node at(4,0){$\bullet$};
\draw(0,0)--(2,0); \draw(3,0)--(4,0); \draw[dotted](2,0)--(3,0);
\end{tikzpicture}}
& $A_1^{(1)}$ &
\begin{tikzpicture}[baseline=0mm]
\node at(0,0){$\bullet$}; \node at(1,0){$\bullet$};
\draw(0,0)--node[above]{$4$}(1,0); \node at(0,0.5){};\node at(0,-0.2){};
\end{tikzpicture}
\\\cline{3-4}
& & $A_n^{(1)}$ &
\begin{tikzpicture}[baseline=3mm,scale=0.7]
\node at(0,0){$\bullet$}; \node at(1,0){$\bullet$}; \node at(4,0){$\bullet$}; \node at(2,1){$\bullet$};
\draw(4,0)--(2,1)--(0,0)--(2,0); \draw(3,0)--(4,0); \draw[dotted](2,0)--(3,0); \node at(0,1.5){};\node at(0,-0.5){};
\end{tikzpicture}
\\\hline
%%% B_n
\multirow{1}{*}{$\begin{matrix}B_n \\\rotatebox{90}{=} \\ C_n \end{matrix}$} &
\multirow{2.5}{*}{\begin{tikzpicture}[baseline=0mm,scale=0.7]
\node at(0,0){$\bullet$}; \node at(1,0){$\bullet$}; \node at(4,0){$\bullet$};
\draw(0,0)--node[above]{$2$}(1,0)--(2,0); \draw[-](3,0)--(4,0); \draw[dotted](2,0)--(3,0);
\end{tikzpicture}}
& $B_n^{(1)}$ &
\begin{tikzpicture}[baseline=-1mm,scale=0.7]
\node at(-1,0){$\bullet$}; \node at(0,0){$\bullet$}; \node at(3,0){$\bullet$}; \node at(3.7,0.7){$\bullet$}; \node at(3.7,-0.7){$\bullet$};
\draw(-1,0)--node[above]{$2$}(0,0)--(1,0); \draw[dotted](1,0)--(2,0); \draw(2,0)--(3,0)--(3.7,0.7) (3,0)--(3.7,-0.7);
\node at(0,1){};\node at(0,-1){};
\end{tikzpicture}
\\\cline{3-4}
 & & $C_n^{(1)}$ &
\begin{tikzpicture}[baseline=-0mm,scale=0.7]
\node at(0,0){$\bullet$}; \node at(1,0){$\bullet$}; \node at(4,0){$\bullet$};
\draw(0,0)--node[above]{$2$}(1,0)--(2,0); \draw[-](3,0)--node[above]{$2$}(4,0); \draw[dotted](2,0)--(3,0);
\node at(0,0.7){};\node at(0,-0.3){};
\end{tikzpicture}
\\\hline%%% D_n
$D_n$ &
\begin{tikzpicture}[baseline=-1mm,scale=0.7]
\node at(-1,0){$\bullet$}; \node at(0,0){$\bullet$}; \node at(3,0){$\bullet$}; \node at(3.7,0.7){$\bullet$}; \node at(3.7,-0.7){$\bullet$};
\draw(-1,0)--(0,0)--(1,0); \draw[dotted](1,0)--(2,0); \draw(2,0)--(3,0)--(3.7,0.7) (3,0)--(3.7,-0.7);
\end{tikzpicture}
& $D_n^{(1)}$ &
\begin{tikzpicture}[baseline=-1mm,scale=0.7]
\node at(0,0){$\bullet$}; \node at(1,0){$\bullet$}; \node at(4,0){$\bullet$}; \node at(4.7,0.7){$\bullet$}; \node at(4.7,-0.7){$\bullet$};
\node at(-0.7,0.7){$\bullet$}; \node at(-0.7,-0.7){$\bullet$};
\draw(0,0)--(1,0)--(2,0); \draw[dotted](2,0)--(3,0); \draw(3,0)--(4,0)--(4.7,0.7) (4,0)--(4.7,-0.7); \draw(-0.7,0.7)--(0,0)--(-0.7,-0.7);
\node at(0,1){};\node at(0,-1){};
\end{tikzpicture}
\\\hline%%% E_6
$E_6$ &
\begin{tikzpicture}[baseline=3mm,scale=0.7]
\node at(-2,0){$\bullet$}; \node at(-1,0){$\bullet$}; \node at(0,0){$\bullet$};
\node at(1,0){$\bullet$}; \node at(2,0){$\bullet$}; \node at(0,1){$\bullet$};
\draw(-2,0)--(2,0); \draw(0,0)--(0,1);
\end{tikzpicture}
& $E_6^{(1)}$ &
\begin{tikzpicture}[baseline=7mm,scale=0.7]
\node at(-2,0){$\bullet$}; \node at(-1,0){$\bullet$}; \node at(0,0){$\bullet$};
\node at(1,0){$\bullet$}; \node at(2,0){$\bullet$}; \node at(0,1){$\bullet$}; \node at(0,2){$\bullet$};
\draw(-2,0)--(2,0); \draw(0,0)--(0,2); \node at(0,2.3){};\node at(0,-0.3){};
\end{tikzpicture}
\\\hline%%% E_7
$E_7$ &
\begin{tikzpicture}[baseline=3mm,scale=0.7]
\node at(-2,0){$\bullet$}; \node at(-1,0){$\bullet$}; \node at(0,0){$\bullet$};
\node at(1,0){$\bullet$}; \node at(2,0){$\bullet$}; \node at(3,0){$\bullet$}; \node at(0,1){$\bullet$};
\draw(-2,0)--(3,0); \draw(0,0)--(0,1);
\end{tikzpicture}
& $E_7^{(1)}$ &
\begin{tikzpicture}[baseline=3mm,scale=0.65]
\node at(-2,0){$\bullet$}; \node at(-1,0){$\bullet$}; \node at(0,0){$\bullet$};
\node at(1,0){$\bullet$}; \node at(2,0){$\bullet$}; \node at(3,0){$\bullet$}; \node at(0,1){$\bullet$}; \node at(-3,0){$\bullet$};
\draw(-3,0)--(3,0); \draw(0,0)--(0,1); \node at(0,1.3){};\node at(0,-0.3){};
\end{tikzpicture}
\\\hline%%% E_8
$E_8$ &
\begin{tikzpicture}[baseline=6mm,scale=0.65]
\node at(-1,0){$\bullet$}; \node at(0,0){$\bullet$}; \node at(0,1){$\bullet$}; \node at(0,2){$\bullet$};
\node at(1,0){$\bullet$}; \node at(2,0){$\bullet$}; \node at(3,0){$\bullet$}; \node at(4,0){$\bullet$};
\draw(-1,0)--(4,0); \draw(0,0)--(0,2);
\end{tikzpicture}
& $E_8^{(1)}$ &
\begin{tikzpicture}[baseline=6mm,scale=0.6]
\node at(-1,0){$\bullet$}; \node at(0,0){$\bullet$}; \node at(0,1){$\bullet$}; \node at(0,2){$\bullet$};
\node at(1,0){$\bullet$}; \node at(2,0){$\bullet$}; \node at(3,0){$\bullet$}; \node at(4,0){$\bullet$}; \node at(5,0){$\bullet$};
\draw(-1,0)--(5,0); \draw(0,0)--(0,2); \node at(0,2.3){};\node at(0,-0.3){};
\end{tikzpicture}
\\\hline%% F_4
$F_4$ &
\begin{tikzpicture}[baseline=0mm]
\node at(0,0){$\bullet$}; \node at(1,0){$\bullet$}; \node at(2,0){$\bullet$}; \node at(3,0){$\bullet$};
\draw(0,0)--(1,0)--node[above]{$2$}(2,0)--(3,0);
\end{tikzpicture}
& $F_4^{(1)}$ &
\begin{tikzpicture}[baseline=0mm,scale=0.9]
\node at(0,0){$\bullet$}; \node at(1,0){$\bullet$}; \node at(2,0){$\bullet$}; \node at(3,0){$\bullet$}; \node at(4,0){$\bullet$};
\draw(0,0)--(1,0)--node[above]{$2$}(2,0)--(4,0); \node at(0,0.6){};\node at(0,-0.2){};
\end{tikzpicture}
\\\hline%% G_2
$G_2$ &
\begin{tikzpicture}[baseline=0mm]
\node at(0,0){$\bullet$}; \node at(1,0){$\bullet$};
\draw(0,0)--node[above]{$3$}(1,0);
\end{tikzpicture}
& $G_2^{(1)}$ &
\begin{tikzpicture}[baseline=0mm]
\node at(0,0){$\bullet$}; \node at(1,0){$\bullet$}; \node at(2,0){$\bullet$};
\draw(0,0)--node[above]{$3$}(1,0)--(2,0); \node at(0,0.6){};\node at(0,-0.2){};
\end{tikzpicture}
\end{tabular}
   \caption{Dynkin graphs and extended Dynkin graphs, where unspecified weights are one}
   \label{fig:(extended) Dynkin graphs}
\end{figure}
%%%

Felikson, Shapiro and Tumarkin completely classified mutation-finite skew-symmetrizable matrices.

%%%
\begin{theorem}\cite{FeST12a,FeST12b}\label{thm:classify mutation-finite}
A connected skew-symmetrizable matrix is mutation-finite if and only if it is one of the following cases:
\begin{itemize}
\item rank $2$;
\item finite type or affine type;
\item defined from a triangulated orbifold (see \cite{FoST,FeST12c});
\item one of nine types $E_6^{(1,1)}$, $E_7^{(1,1)}$, $E_8^{(1,1)}$, $X_6$, $X_7$, $F_4^{(\ast,+)}$, $F_4^{(\ast,\ast)}$, $G_2^{(\ast,+)}$ or $G_2^{(\ast,\ast)}$ as in Figure \ref{fig:exceptional diagrams}.
\end{itemize}
\end{theorem}

%%%
\begin{figure}[htp]
\begin{tabular}{c|c|c|c}
$E_6^{(1,1)}$ &
\begin{tikzpicture}[baseline=0mm,scale=0.8]
\node(l1)at(-0.5,0){$\bullet$}; \node(l2)at(-1.5,0){$\bullet$}; 
\node(u)at(0,1){$\bullet$}; \node(d)at(0,-1){$\bullet$};
\node(r1)at(0.5,0){$\bullet$}; \node(r2)at(1.5,0){$\bullet$}; 
\node(r3)at(2.5,0){$\bullet$}; \node(r4)at(3.5,0){$\bullet$};
\draw[<-](l2)--(l1); \draw[->](u)--(l1); \draw[->](l1)--(d); \draw[->](u)--(r1); \draw[->](r1)--(d);
\draw[->](-0.06,-0.7)--(-0.06,0.7); \draw[->](0.06,-0.7)--(0.06,0.7);
\draw[->](r2)--(r1); \draw[->](r4)--(r3); \draw[->](u)--(r3); \draw[->](r3)--(d);
\end{tikzpicture}
& $X_6$ &
\begin{tikzpicture}[baseline=0mm,scale=0.8]
\node(0)at(0,0){$\bullet$}; \node(l)at(180:1){$\bullet$}; \node(lu)at(120:1){$\bullet$};
\node(r)at(0:1){$\bullet$}; \node(ru)at(60:1){$\bullet$}; \node(d)at(0,-1){$\bullet$};
\draw[->](0)--(l); \draw[->](lu)--(0); \draw[->](0)--(ru); \draw[->](r)--(0); \draw[->](d)--(0);
\draw[->](170:1)--(130:1); \draw[->](173:0.9)--(127:0.9);
\draw[->](50:1)--(10:1); \draw[->](53:0.9)--(7:0.9); \node at(0,1.3){};\node at(0,-1.3){};
\end{tikzpicture}
\\\hline
$E_7^{(1,1)}$ &
\begin{tikzpicture}[baseline=0mm,scale=0.8]
\node(l1)at(-0.5,0){$\bullet$}; \node(l2)at(-1.5,0){$\bullet$}; \node(l3)at(-2.5,0){$\bullet$}; 
\node(u)at(0,1){$\bullet$}; \node(d)at(0,-1){$\bullet$};
\node(r1)at(0.5,0){$\bullet$}; \node(r2)at(1.5,0){$\bullet$}; 
\node(r3)at(2.5,0){$\bullet$}; \node(r4)at(3.5,0){$\bullet$};
\draw[<-](l3)--(l2); \draw[<-](l2)--(l1); \draw[->](u)--(l1); \draw[->](l1)--(d); \draw[->](u)--(r1); \draw[->](r1)--(d);
\draw[->](-0.06,-0.7)--(-0.06,0.7); \draw[->](0.06,-0.7)--(0.06,0.7);
\draw[->](r3)--(r2); \draw[->](r4)--(r3); \draw[->](u)--(r2); \draw[->](r2)--(d);
\end{tikzpicture}
& $X_7$ &
\begin{tikzpicture}[baseline=0mm,scale=0.8]
\node(0)at(0,0){$\bullet$}; \node(l)at(180:1){$\bullet$}; \node(lu)at(120:1){$\bullet$};
\node(r)at(0:1){$\bullet$}; \node(ru)at(60:1){$\bullet$};
\node(dl)at(-120:1){$\bullet$}; \node(dr)at(-60:1){$\bullet$};
\draw[->](0)--(l); \draw[->](lu)--(0); \draw[->](0)--(ru); \draw[->](r)--(0); \draw[->](0)--(dr); \draw[->](dl)--(0);
\draw[->](170:1)--(130:1); \draw[->](173:0.9)--(127:0.9);
\draw[->](50:1)--(10:1); \draw[->](53:0.9)--(7:0.9);
\draw[->](-70:1)--(-110:1); \draw[->](-67:0.9)--(-113:0.9); \node at(0,1.3){};\node at(0,-1.3){};
\end{tikzpicture}
\\\hline
$E_8^{(1,1)}$ &
\begin{tikzpicture}[baseline=0mm,scale=0.7]
\node(l1)at(-0.5,0){$\bullet$}; \node(l2)at(-1.5,0){$\bullet$};
\node(u)at(0,1){$\bullet$}; \node(d)at(0,-1){$\bullet$};
\node(r1)at(0.5,0){$\bullet$}; \node(r2)at(1.5,0){$\bullet$}; 
\node(r3)at(2.5,0){$\bullet$}; \node(r4)at(3.5,0){$\bullet$};
\node(r5)at(4.5,0){$\bullet$}; \node(r6)at(5.5,0){$\bullet$};
\draw[<-](l2)--(l1); \draw[->](u)--(l1); \draw[->](l1)--(d); \draw[->](u)--(r1); \draw[->](r1)--(d);
\draw[->](-0.06,-0.7)--(-0.06,0.7); \draw[->](0.06,-0.7)--(0.06,0.7);
\draw[->](r3)--(r2); \draw[->](r4)--(r3); \draw[->](u)--(r2); \draw[->](r2)--(d);
\draw[->](r5)--(r4); \draw[->](r6)--(r5); \node at(0,1.3){};\node at(0,-1.3){};
\end{tikzpicture}
\\\hline
$F_4^{(\ast,+)}$ &
\begin{tikzpicture}[baseline=0mm,scale=0.75]
\node(l2)at(-1,0){$\bullet$}; \node(l1)at(0,0){$\bullet$}; \node(l)at(1,0){$\bullet$};
\node(u)at(2,0.8){$\bullet$}; \node(d)at(2,-0.8){$\bullet$}; \node(r)at(3,0){$\bullet$};
\draw[<-](l2)--(l1); \draw[<-](l1)--(l); \draw[<-](l)--node[above left]{$2$}(u); \draw[<-](u)--node[right]{$4$}(d);
\draw[<-](d)--node[below left]{$2$}(l); \draw[<-](d)--(r); \draw[<-](r)--(u);
\end{tikzpicture}
& $G_2^{(\ast,+)}$ &
\begin{tikzpicture}[baseline=0mm,scale=0.8]
\node(l1)at(0,0){$\bullet$}; \node(l)at(1,0){$\bullet$}; \node(u)at(2,0.8){$\bullet$}; \node(d)at(2,-0.8){$\bullet$};
\draw[<-](l1)--(l); \draw[<-](l)--node[above left]{$3$}(u); \draw[<-](u)--node[right]{$4$}(d); \draw[<-](d)--node[below left]{$3$}(l); 
\node at(0,1.1){}; \node at(0,-1.1){};
\end{tikzpicture}
\\\hline
$F_4^{(\ast,\ast)}$ &
\begin{tikzpicture}[baseline=0mm,scale=0.75]
\node(l1)at(0,0){$\bullet$}; \node(l)at(1,0){$\bullet$};
\node(u)at(2,0.8){$\bullet$}; \node(d)at(2,-0.8){$\bullet$}; \node(r)at(3,0){$\bullet$}; \node(r1)at(4,0){$\bullet$}; 
\draw[<-](l1)--(l); \draw[<-](l)--node[above left]{$2$}(u); \draw[<-](u)--node[right]{$4$}(d);
\draw[<-](d)--node[below left]{$2$}(l); \draw[<-](d)--(r); \draw[<-](r)--(u); \draw[<-](r)--(r1); \node at(0,1.1){}; \node at(0,-1.1){};
\end{tikzpicture}
& $G_2^{(\ast,\ast)}$ &
\begin{tikzpicture}[baseline=0mm,scale=0.8]
\node(l)at(1,0){$\bullet$}; \node(u)at(2,0.8){$\bullet$}; \node(d)at(2,-0.8){$\bullet$}; \node(r)at(3,0){$\bullet$};
\draw[<-](l)--node[above left]{$3$}(u); \draw[<-](u)--node[right]{$4$}(d); \draw[<-](d)--node[below left]{$3$}(l);
\draw[<-](d)--(r); \draw[<-](r)--(u);
\end{tikzpicture}
\end{tabular}
   \caption{Exceptional quivers and diagrams}
   \label{fig:exceptional diagrams}
\end{figure}
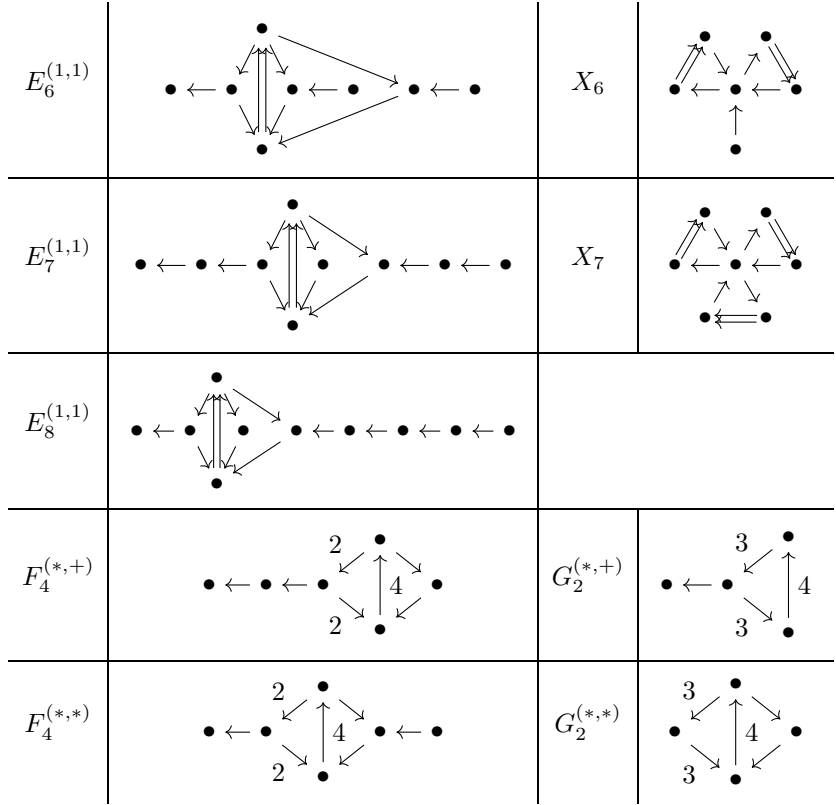
%%%

\begin{remark}\label{rem:mut-finquivers}
\begin{enumerate}
\item The quivers $E_6^{(1,1)}$, $E_7^{(1,1)}$ and $E_8^{(1,1)}$ are called \emph{tubular quivers} and studied in e.g. \cite{BG,BGJ,GG}.
\item Derksen and Owen \cite{DO} found the quivers $X_6$ and $X_7$ as new mutation-finite quivers.
\item The diagrams $F_4^{(\ast,+)}$, $F_4^{(\ast,\ast)}$, $G_2^{(\ast,+)}$ and $G_2^{(\ast,\ast)}$ appear as diagrams of
extended affine root systems (see \cite{FeST12b,Sa}).
\end{enumerate}
\end{remark}

%%% Cluster algebras
%%%
\subsection{Cluster algebras}\label{subsec:cluster algebras}

We recall the definition of cluster algebras for the special case, that is, cluster algebras with principal coefficients \cite{FZ07}. We only study this case and refer to \cite{FZ02,FZ07} for a general definition of cluster algebras. Before giving the definition, we prepare some notations. Let $\cF$ be the field of rational functions in $2n$ variables over $\bQ$.

\begin{definition}\label{def:seed mutation}
\begin{enumerate}
\item A \emph{seed with coefficients} is a pair $(\bx,C)$ consisting of the following data:
\begin{enumerate}
\item $\bx=(x_1,\ldots,x_n,y_1=x_{n+1},\ldots,y_n=x_{2n})$ is a free generating set of $\cF$ over $\bQ$.
\item $C=(c_{ij})_{i \in I, 1\le j \le n}$ is a $2n \times n$ integer matrix whose upper part $(c_{ij})_{1 \le i,j \le n}$ is skew-symmetrizable, where $I:=\{1,\ldots,2n\}$.
\end{enumerate}
Then we refer to $\bx$ as the \emph{cluster}, to each $x_i$ ($1 \le i \le n$) as a \emph{cluster variable} and $y_i$ as a \emph{coefficient}.
\item For a seed $(\bx,C)$ with coefficients, the \emph{mutation $\mu_k(\bx,C)=((x'_1,\ldots,x'_n,y_1,\ldots,y_n),\mu_k C)$ at $k$} $(1 \le k \le n)$ is defined by $x'_i = x_i$ if $i \neq k$, and
\[
x_k x'_k = \prod_{i \in I}x_i^{[c_{ik}]_+}+\prod_{i \in I}x_i^{[-c_{ik}]_+}.
\]
\end{enumerate}
\end{definition}

Note that $\mu_k$ is an involution and $\mu_k(\bx,C)$ is also a seed with coefficients. Moreover, the coefficients are not changed by mutations.

The \emph{principal extension of $B$} is the skew-symmetrizable $2n \times n$ matrix $\hat{B}$ whose upper part is $B$ and whose lower part is the identity matrix of rank $n$. We fix a seed $(\bx=(x_1,\ldots,x_n,y_1,\ldots,y_n),\hat{B})$ with coefficients, called the \emph{initial seed}. We also call each $x_i$ the \emph{initial cluster variable}. We are ready to define cluster algebras with principal coefficients.

%%% def of cluster algebras
\begin{definition}
The \emph{cluster algebra $\cA(B)=\cA(\bx,\hat{B})$ with principal coefficients} for the initial seed $(\bx,\hat{B})$ is a $\bZ[y_1^{\pm},\ldots,y_n^{\pm}]$-subalgebra of $\cF$ generated by the cluster variables obtained from $(\bx,\hat{B})$ by all sequences of mutations.
\end{definition}

We say that the cluster algebra $\cA(B)$ is \emph{finite type} (resp., \emph{affine type}, \emph{type $\Gamma$} for a diagram $\Gamma$, connected) if so is $B$.

One of the remarkable properties of cluster algebras with principal coefficients is the (strongly) \emph{Laurent phenomenon}.

\begin{theorem}\cite{FZ02,FZ07}\label{thm:Laurent phenomenon}
Any cluster variable $x$ is expressed by a Laurent polynomial of the initial cluster variables $x_1,\ldots,x_n$ and coefficients $y_1,\ldots,y_n$ with form
\[
x=\frac{f(x_1,\ldots,x_n,y_1,\ldots,y_n)}{x_1^{d_1} \cdots x_n^{d_n}},
\]
where $f(x_1,\ldots,x_n,y_1,\ldots,y_n) \in \bZ[x_1,\ldots,x_n,y_1,\ldots,y_n]$ and $d_i \in \bZ_{\ge 0}$.
\end{theorem}

\begin{example}\label{exam:cluster algebras}
We consider a type $B_2$ matrix $B_{2,1}$. The mutation of the initial seed $(\bx,\hat{B}_{2,1})$ at $1$ is given by
\[
\left((x_1,x_2,y_1,y_2), \begin{bmatrix}0 & 1 \\ -2 & 0 \\ 1 & 0 \\ 0 & 1 \end{bmatrix}\right) \mapsto 
\left(\left(\frac{x_2^2+y_1}{x_1},x_2,y_1,y_2\right), \begin{bmatrix}0 & -1 \\ 2 & 0 \\ -1 & 1 \\ 0 & 1 \end{bmatrix}\right).
\]
Repeating mutations, we get the cluster algebra
{\setlength\arraycolsep{0mm}
\begin{eqnarray*}
\cA(B_{2,1})=\bZ\biggl[&&y_1^{\pm},y_2^{\pm},x_1,x_2,\cfrac{x_2^2+y_1}{x_1},\cfrac{x_2^2+y_1y_2x_1+y_1}{x_1x_2},\\
&&\cfrac{y_1y_2^2x_1^2+x_2^2+2y_1y_2x_1+y_1}{x_1x_2^2},\cfrac{y_2x_1+1}{x_2}\biggr].
\end{eqnarray*}}
\end{example}

The definition of cluster algebras naturally induces an isomophism $\cA(B) \simeq \cA(B') \otimes_{\bZ} \cA(B'')$ for skew-symmetrizable matrices $B$, $B'$ and $B''$ with $\Gamma(B)=\Gamma(B') \sqcup \Gamma(B'')$.

In the rest of this section, we assume that $B$ is connected.

%%% $g$-vector fans of cluster algebras
%%%
\subsection{$g$-vector fans of cluster algebras}\label{subsec:g-vector fans}

Laurent phenomenon (Theorem \ref{thm:Laurent phenomenon}) means that $\cA(B)$ is contained in $\bZ[x_1^{\pm 1},\ldots,x_n^{\pm 1},y_1,\ldots,y_n]$. We consider its $\bZ^n$-grading given by
\[
 \deg(x_i)=\be_i,\ \ \deg(y_j)=\sum_{i=1}^n -b_{ij}\be_i,
\]
where $\be_1,\ldots,\be_n$ are the standard basis vectors in $\bZ^n$.

\begin{proposition}\cite[Proposition 6.1]{FZ07}
Every cluster variable $x$ of $\cA(B)$ is homogeneous with respect to the $\bZ^n$-grading $\deg$.
\end{proposition}

\begin{definition}
The degree $\deg(x)$ of a cluster variable $x$ is called the \emph{$g$-vector} of $x$. The cone spanned by the $g$-vectors of cluster variables in a cluster $\bx$ is called the \emph{$g$-vector cone} of $\bx$, that is, it is given by $\{\sum_{x \in \bx}a_x \deg(x) \mid a_x \in \bR_{\ge 0}\}$.
\end{definition}

\begin{remark}
For a seed  $((z_1,\ldots,z_n,y_1,\ldots,y_n),(c_{ij}))$ of $\cA(B)$, we denote by $z_k'$ the new cluster variable obtained by mutating it at $k$. Then the mutation rule of cluster variables as in Definition \ref{def:seed mutation}(2) gives the recurrence relation (see \cite[Proposition 6.6]{FZ07})
{\setlength\arraycolsep{0mm}
\begin{eqnarray*}
\deg(z_k')&=&-\deg(z_k)+\sum_{i=1}^n[c_{ik}]_+\deg(z_i)+\sum_{i=1}^n[c_{n+i,k}]_+\deg(y_i)\\
&=&-\deg(z_k)+\sum_{i=1}^n[-c_{ik}]_+\deg(z_i)+\sum_{i=1}^n[-c_{n+i,k}]_+\deg(y_i).
\end{eqnarray*}}
All $g$-vectors of $\cA(B)$ are uniquely determined by the initial condition $\deg(x_i)=\be_i$ for any $i$ and this recurrence relation.
\end{remark}

One of important properties of $g$-vectors, called sign-coherence property, is the following.

\begin{theorem}\cite[Theorem 5.11]{GHKK}\label{thm:sign-coherence}
For any cluster $(z_1,\ldots,z_n,$ $y_1,\ldots,y_n)$ and $i \in \{1,\ldots,n\}$, the $i$-th coordinates of the $g$-vectors $\deg(z_1),\ldots,\deg(z_n)$ are either all non-negative or all non-positive.
\end{theorem}

Theorem \ref{thm:sign-coherence} was firstly proved in \cite{DWZ10} for skew-symmetric cluster algebras (see also \cite{LS}). Also, we can obtain it via additive categorification of skew-symmetric cluster algebras (see Section \ref{sec:rep theory}). Nakanishi and Zelevinsky \cite{NZ} showed that many of conjectures in \cite{FZ07} follow from the sign-coherence property. The following is one of them.

\begin{theorem}\cite[Corollary 5.5]{GHKK},\cite[Proposition 4.2]{NZ}\label{thm:linearly independent}
The $g$-vectors of cluster variables in a cluster are linearly independent. Thus $g$-vector cones in $\cA(B)$ have dimension $n$.
\end{theorem}

All $g$-vector cones in a cluster algebra and their faces form a fan, which is the main subject in this paper.

%%%
\begin{theorem}\cite[Theorem 0.8]{GHKK}\label{thm:cluster g-vector fan}
There is a simplicial fan each of whose 
\begin{itemize}
\item ray is spanned by the $g$-vector of a cluster variable of $\cA(B)$;
\item maximal cone is the $g$-vector cone of a cluster of $\cA(B)$;
\end{itemize}
\end{theorem}

\begin{definition}
The fan in Theorem \ref{thm:cluster g-vector fan} is called the \emph{$g$-vector fan of $\cA(B)$}, and we denote it by $\cF(B)$. We denote by $\overline{\cF(B)}$ the closure of the union of cones of $\cF(B)$.
\end{definition}

The notation of $g$-vector fans also appear in representation theory of algebras (Theorem \ref{thm:silting-gvector-fan} and Definition \ref{def:silting-gvector-fan}). Via categorification, the $g$-vector fan of a cluster algebra coincides with a subfan of the $g$-vector fan of the corresponding algebra (Theorem \ref{thm:categorification}).

%%% example rank 2 g-vector fans
\begin{example}\label{exam:rank 2 g-vector fan}
Example \ref{exam:cluster algebras} gives the $g$-vector fan of $\cA(B_{2,1})$ as in Figure \ref{fig:rank 2 g-vector fans}.
In fact, for rank $2$ cluster algebras, their $g$-vector fans are well-known. In particular, for $bc \ge 4$, $\cF(B_{b,c})$ contains infinitely many rays converging to the rays $r_{\pm}$ of slope $(-bc \pm \sqrt{bc(bc-4)})/2c$. Moreover, if $bc=4$, then $r_+=r_-$. If $bc > 4$, then $r_+ \neq r_-$ and the interior of the cone spanned by $r_+$ and $r_-$ is the complement of $\overline{\cF(B_{b,c})}$ (see e.g. \cite[Example 1.15]{GHKK}). For example, $\cF(B_{4,1})$ and $\cF(B_{5,1})$ are given as in Figure \ref{fig:rank 2 g-vector fans}.
\end{example}
%%%
\begin{figure}[htp]
\centering
\begin{tabular}{ccc}
\begin{tikzpicture}[baseline=-1mm,scale=1.5]
\coordinate(0)at(0,0); \coordinate(x)at(1,0); \coordinate(-x)at(-1,0); \coordinate(y)at(0,1); \coordinate(-y)at(0,-1);
\draw[->](-x)--(x); \draw[->](-y)--(y); \draw(-0.5,1)--(0)--(-1,1);
\end{tikzpicture}
&
\begin{tikzpicture}[baseline=-1mm,scale=1.5]
\coordinate(0)at(0,0); \coordinate(x)at(1,0); \coordinate(-x)at(-1,0); \coordinate(y)at(0,1); \coordinate(-y)at(0,-1);
\draw[->](-x)--(x); \draw[->](-y)--(y); \draw(-0.25,1)--(0)--(-1,1); \draw(-0.33,1)--(0)--(-0.75,1); \draw(-0.375,1)--(0)--(-0.66,1);
 \draw(-0.4,1)--(0)--(-0.6,1); \draw(-0.42,1)--(0)--(-0.57,1); \draw(-0.44,1)--(0)--(-0.55,1); \draw(-0.46,1)--(0)--(-0.53,1);
\draw[red](0)--(-0.5,1); \node at(-0.5,1.15){$r_-=r_+$};
\end{tikzpicture}
&
\begin{tikzpicture}[baseline=-1mm,scale=1.5]
\coordinate(0)at(0,0); \coordinate(x)at(1,0); \coordinate(-x)at(-1,0); \coordinate(y)at(0,1); \coordinate(-y)at(0,-1);
\draw[->](-x)--(x); \draw[->](-y)--(y); \draw(-0.2,1)--(0)--(-1,1); \draw(-0.25,1)--(0)--(-0.8,1);
\draw(-0.267,1)--(0)--(-0.75,1); \draw(-0.2727,1)--(0)--(-0.733,1);
\draw[red](-0.2764,1)--(0)--(-0.7236,1); \node at(-0.7236,1.15){$r_-$}; \node at(-0.2764,1.15){$r_+$};
\end{tikzpicture}\\
$\cF(B_{2,1})$ & $\cF(B_{4,1})$ & $\cF(B_{5,1})$
\end{tabular}
   \caption{Examples of the $g$-vector fans $\cF(B_{b,c})$}
   \label{fig:rank 2 g-vector fans}
\end{figure}
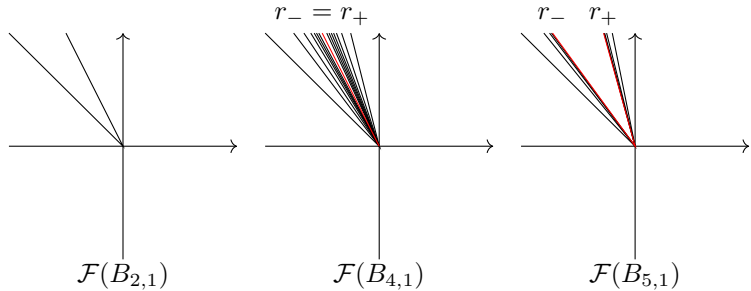
%%%

Finally, we recall the following transition rule, which was conjectured in \cite[Conjecture 7.12]{FZ07}.

\begin{theorem}\cite[Corollary 5.5]{GHKK}\cite[Proposition 4.2]{NZ}\label{thm:linear transformation}
For $k \in \{1,\ldots,n\}$, $\cF(\mu_k B)$ is obtained from $\cF(B)$ by the map $(g_i)_{1 \le i \le n} \mapsto (g_i')_{1 \le i \le n}$, where
\[
g_i' = \left\{\begin{array}{ll}
 -g_k & \mbox{if} \ \ i=k,\\
 g_i + [b_{ik}]_+g_k - b_{ik}\min(g_k,0) & \mbox{otherwise}.
\end{array} \right.
\]
\end{theorem}

%%% Cluster algebras with dense $g$-vector fans
%%%
\subsection{Cluster algebras with dense $g$-vector fans}\label{subsec:dense g-vector fans}

The $g$-vector fans of finite type cluster algebras have several nice properties. We focus on a basic property (see e.g. \cite[Theorem 10.6]{R} and proof of \cite[Proposition 4.9]{A}).

\begin{theorem}\cite{A,R}\label{thm:finite type complete}
The matrix $B$ is finite type if and only if the $g$-vector fan $\cF(B)$ is complete, that is, the union of its cones covers $\bR^n$.
\end{theorem}

This naturally leads the following definition in a general setting.

\begin{definition}
We say that the $g$-vector fan $\cF(B)$ is \emph{dense} if the union of its cones is dense in $\bR^n$.
\end{definition}

By Example \ref{exam:rank 2 g-vector fan}, we know that $\cF(B_{b,c})$ is dense (resp., complete) if and only if $bc \le 4$ (resp., $bc < 4$).

The denseness of $g$-vector fans is preserved by mutations.

\begin{proposition}\label{prop:preserve denseness}
If $\cF(B)$ is dense, then so is $\cF(\mu_k B)$.
\end{proposition}

\begin{proof}
The assertion holds since the transformation of mutations for $g$-vectors is piecewise linear by Theorem \ref{thm:linear transformation}.
\end{proof}

Reading and Speyer \cite{RS} studied the $g$-vector fans $\cF(B)$ of affine type skew-symmetrizable matrix $B$ and explicitly described the complement of $\cF(B)$ as a cone of co-dimension one. In particular, this gives the denseness of $\cF(B)$.

%%%
\begin{theorem}\cite[Corollaries 1.3 and 4.9]{RS}\label{thm:affine type dense}
If $B$ is affine type, then the $g$-vector fan $\cF(B)$ is dense.
\end{theorem}

In general, it is not easy to exhibit the complement of a $g$-vector fan. However, it'll be possible to see the denseness to a certain extent. Skew-symmetric cluster algebras with dense $g$-vector fans were classified in \cite{PY,Se14} except for type $X_6$. In which case, we still don't know if the associated $g$-vector fan is dense, but we conjecture that it is dense (see \cite[Conjecture 1.3]{PY}).

%%%
\begin{theorem}\cite{PY,Se14}\label{thm:classify g-dense cluster}
Suppose that $B$ is skew-symmetric and not of type $X_6$.
Then the $g$-vector fan $\cF(B)$ is dense if and only if $B$ is mutation-finite and none of the following conditions are satisfied:
\begin{itemize}
\item $B_{b,b}$ for $b \in \bZ$ with $|b|>2$;
\item defined from a triangulated surface with exactly one puncture;
\item type $X_7$.
\end{itemize}
\end{theorem}

\begin{proof}
 Except for type $X_7$, the assertion follows from \cite[Theorem 5.17]{PY}. Moreover, Seven \cite{Se14} showed that if $B$ is type $X_7$, then $\cF(B)$ is contained in some open half-space in $\bR^7$, in particular, it is not dense.
\end{proof}

\cite[Theorem 5.17]{PY} was proved via an additive categorification of skew-symmetric cluster algebras (see Section \ref{sec:rep theory}). It is an open problem to give a categorification of skew-symmetrizable cluster algebras.

%%% Acyclic cluster algebras with dense $g$-vector fans
%%%
\subsection{Infinite mutation type}\label{subsec:infinite mutation type}

The following property was used in \cite{PY} to prove Theorem \ref{thm:classify g-dense cluster} for skew-symmetric case. It is naturally extended to skew-symmetrizable case, and we will give its proof for the convenience of the reader.

\begin{theorem}\label{thm:mutation-infinite g-dense}
If $\cF(B)$ is dense, then $B$ is mutation-finite.
\end{theorem}

To prove Theorem \ref{thm:mutation-infinite g-dense}, we use a pull back of scattering diagrams. We refer to \cite{B,GHKK,M} for the details of scattering diagrams. Roughly speaking, a \emph{scattering diagram} is a set of \emph{walls}, where a wall is a cone of co-dimension one in $\bR^n$ together with some function. For the union of walls in a scattering diagram $\fD$, a connected component of its complement is called a \emph{chamber of $\fD$}. One can construct a scattering diagram $\fD(B)$ associated with $B$ and it relates to the cluster algebra $\cA(B)$. We only state their properties which we need in this paper.

%%%
\begin{theorem}\cite[Theorem 0.8]{GHKK}\label{thm:g-vector cone=chamber}
The interior of a $g$-vector cone in $\cA(B)$ is a chamber of $\fD(B)$.
\end{theorem}

Theorem \ref{thm:g-vector cone=chamber} means that a mutation in $\cA(B)$ corresponds an adjacent pair of chambers $C$ and $C'$ of $\fD(B)$. We say that $C'$ is the \emph{mutation of $C$ at the wall $\overline{C} \cap \overline{C'}$}. Thus any $g$-vector cone in $\cA(B)$ is the closure of a chamber $C$ of $\fD(B)$ obtained from $C_0^{B}$ by a finite sequence of mutations, where $C_0^{B}$ is the interior of the cone spanned by $\be_1,\ldots,\be_n$,

The theorem below follows from the pull back of scattering diagrams given by Muller \cite[Theorem 33]{M} (see also the proof of \cite[Theorem 4.8]{CL}). For a subset $I \subset \{1,\ldots,n\}$, we consider a projection $\pi_I : \bR^n \rightarrow \bR^{|I|}$ given by $(r_i)_{1 \le i \le n} \mapsto (r_i)_{i \in I}$. We denote by $B_I$ the principal submatrix of $B$ indexed by $I$.

%%%
\begin{theorem}\cite{CL,M}\label{prop:pull back}
Let $\pi_I^{\ast}\fD(B_I)$ be a scattering diagram consisting of the walls $\pi_I^{-1}W$ for all walls $W$ of $\fD(B_I)$. Then each chamber of $\fD(B)$ is contained in some chamber of $\pi_I^{\ast}\fD(B_I)$.
\end{theorem}

This confirms Theorem \ref{thm:mutation-infinite g-dense} in a certain case.

%%%
\begin{lemma}\label{lem:rank2 not dense}
If there is $J \subset \{1,\ldots,n\}$ such that $B_J=B_{b,c}$ with $bc > 4$, then the $g$-vector fan $\cF(B)$ is not dense.
\end{lemma}

\begin{proof}
Let $C$ be the interior of a $g$-vector cone in $\cA(B)$, which is a chamber of $\fD(B)$ by Theorem \ref{thm:g-vector cone=chamber}, and $C'$ be a mutation of $C$. By Theorem \ref{prop:pull back}, $\pi_J C$ is contained in some chamber of $\fD(B_J)$ and $\pi_J C'$ is contained in the same chamber or its adjacent chambers. Therefore, since $\pi_J C_0^B=C_0^{B_J}$ and there are infinitely many chambers of $\fD(B_J)$ converging to $r_{\pm}$ as in Example \ref{exam:rank 2 g-vector fan}, $\fD(B)$ has no intersection with the cone in $\bR^n$ spanned by $\pi_J^{-1}r_{\pm}$. Since the cone has dimension $n$, $\fD(B)$ is not dense.
\end{proof}

We are ready to prove Theorem \ref{thm:mutation-infinite g-dense}.

\begin{proof}[Proof of Theorem \ref{thm:mutation-infinite g-dense}]
If $B$ is mutation-infinite, it is mutation equivalent to $B'$ such that $B'_J=B_{b,c}$ with $bc > 4$ for some $J \subset \{1,\ldots,n\}$. By Lemma \ref{lem:rank2 not dense}, $\cF(B')$ is not dense. Therefore, the assertion follows from Proposition \ref{prop:preserve denseness}.
\end{proof}

%%% Acyclic cluster algebras with dense $g$-vector fans
%%%
\subsection{Acyclic cluster algebras with dense $g$-vector fans}\label{subsec:acyclic}

The cluster algebra $\cA(B)$ is called \emph{acyclic} if $B$ is mutation equivalent to an acyclic skew-symmetrizable matrix. A classification of mutation-finite acyclic skew-symmetrizable matrices was given by Seven \cite{Se11} (by Buan and Reiten \cite{BR} for skew-symmetric case).

\begin{theorem}\cite[Theorem 3.5]{Se11}\label{thm:mutation-finite acyclic}
Suppose that $B$ is acyclic. Then $B$ is mutation-finite if and only if it is rank $2$, or $\Gamma(B)$ is either a Dynkin diagram or an extended Dynkin diagram.
\end{theorem}

\begin{proof}[Proof of Theorem \ref{thm:g-dense acyclic}]
Suppose that $B$ is mutation equivalent to an acyclic matrix. By Theorems \ref{thm:mutation-infinite g-dense} and \ref{thm:mutation-finite acyclic}, if $\cF(B)$ is dense, then $B$ must be rank $2$, finite type or affine type. If $B$ is either finite type or affine type, then $\cF(B)$ is dense by Theorems \ref{thm:finite type complete} and \ref{thm:affine type dense}. If $B$ is rank $2$, then $\cF(B)$ is dense if and only if $B=B_{b,c}$ with $bc \le 4$, in which case, it is either finite type or affine type. Therefore, the assertion holds.
\end{proof}

%%% Subsection : Future problems
%%%
\subsection{Future problems}\label{subsec:future problems}

For non-acyclic cluster algebras, we already have some information on the denseness of $g$-vector fans. In particular, Theorem \ref{thm:classify g-dense cluster} gives a complete classification of skew-symmetric matrices with dense $g$-vector fans except for type $X_6$ as in Table \ref{table:skew-symmetric}. On the other hand, for skew-symmetrizable matrices, Table \ref{table:skew-symmetrizable} follows from Theorems \ref{thm:finite type complete}, \ref{thm:affine type dense}, \ref{thm:classify g-dense cluster}, and \ref{thm:mutation-infinite g-dense}.
 \begin{table}[htb]
{\renewcommand{\arraystretch}{1.5}
 \begin{tabular}{c|c|c|c}
   \multicolumn{3}{c|}{Skew-symmetric matrix} & $g$-vector fan\\ \hline\hline
   & surface & once-punctured & non-dense\\ \cline{3-4}
   & type & otherwise & \multirow{3}*{dense} \\ \cline{2-3}
   finite & \multicolumn{2}{c|}{finite, affine type} &\\ \cline{2-3}
   mutation & \multicolumn{2}{c|}{$E_m^{(1,1)}$ ($m=6, 7, 8$)} &\\ \cline{2-4}
   type & \multicolumn{2}{c|}{$X_6$} & ?\\ \cline{2-4}
   & \multicolumn{2}{c|}{$X_7$} & \multirow{3}*{non-dense}\\ \cline{2-3}
   & \multicolumn{2}{c|}{$B_{b,b}$ ($|b| > 2$)} & \\ \cline{1-3}
   \multicolumn{3}{c|}{infinite mutation type} & 
  \end{tabular}}\vspace{2mm}
 \caption{Denseness of $g$-vector fans in skew-symmetric cases}
 \label{table:skew-symmetric}
\end{table}
\begin{table}[htb]
{\renewcommand{\arraystretch}{1.5}
 \begin{tabular}{c|c|c}
   \multicolumn{2}{c|}{Skew-symmetrizable matrix} & $g$-vector fan\\ \hline\hline
   & orbifold type & ?\\ \cline{2-3}
   \multirow{2}*{finite} & finite, affine type & \multirow{2}*{dense} \\ \cline{2-2}
   \multirow{2}*{mutation} & $E_m^{(1,1)}$ ($m=6, 7, 8$) \\ \cline{2-3}
   \multirow{2}*{type} & $X_6$, $F_4^{(\ast,+)}$, $F_4^{(\ast,\ast)}$, $G_2^{(\ast,+)}$, $G_2^{(\ast,\ast)}$ & ?\\ \cline{2-3}
   & $X_7$ & \multirow{3}*{non-dense}\\ \cline{2-2}
   & $B_{b,c}$ ($bc > 4$) & \\ \cline{1-2}
   \multicolumn{2}{c|}{infinite mutation type} & 
  \end{tabular}}\vspace{2mm}
 \caption{Denseness of $g$-vector fans in skew-symmetrizable cases}
 \label{table:skew-symmetrizable}
\end{table}
Therefore, the remaining unknown cases are that a connected skew-symmetrizable matrix $B$ is defined from a triangulated orbifold with at least one orbifold point, or one of types $X_6$, $F_4^{(\ast,+)}$, $F_4^{(\ast,\ast)}$, $G_2^{(\ast,+)}$ or $G_2^{(\ast,\ast)}$. In a follow-up work, we will study for cluster algebras associated with orbifolds.

%%%
%%% Section : $g$-vector fans of finite dimensional algebras
%%%
\section{Application to representation theory}\label{sec:rep theory}

In this section, we introduce a main theorem (Theorem \ref{thm:categorification}) of additive categorification of skew-symmetric cluster algebras. Additive categorification of cluster algebras has been studied in many contexts, for example, decorated representations of quivers with potentials \cite{DWZ08}, cluster categories \cite{P11a,P11b}, and derived categories of Ginzburg differential graded algebras \cite{N}. Here, we deal with it in terms of silting theory. Theorems \ref{thm:g-dense acyclic} and \ref{thm:categorification} will give Theorem \ref{thm:g-tame hereditary}. We fix an algebraically closed field $K$ throughout this section.

%%% Additive categorification of cluster algebras
%%%
\subsection{Quivers with potentials and Jacobian algebras}\label{subsec:QP and Jacobian algebras}

We recall quivers with potentials and their Jacobian algebras which are used to categorify skew-symmetric cluster algebras. We refer to \cite{DWZ08} for the details.

Let $Q$ be a quiver without loops. We denote by $\widehat{KQ}$ the complete path algebra of $Q$ with radical-adic topology. A \emph{potential} $W \in \widehat{KQ}$ of $Q$ is a (possibly infinite) linear combination of oriented cycles in $Q$. The pair $(Q,W)$ is called a \emph{quiver with potential} (QP for short). For an oriented cycle $\alpha_1\cdots\alpha_m$ and an arrow $\alpha$ of $Q$, we define
\[
\partial_{\alpha}(\alpha_1\cdots\alpha_m)=\sum_{i : \alpha_i=\alpha}\alpha_{i+1}\cdots\alpha_m\alpha_1\cdots\alpha_{i-1}.
\]
This is extended to the cyclic derivative $\partial_{\alpha}(W)$ of $W$ by linearity and continuously. The \emph{Jacobian ideal} $J(Q,W)$ is the closure, on radical-adic topology, of the ideal of $\widehat{KQ}$ generated by the set $\{\partial_{\alpha}W \mid \alpha : \text{arrows of $Q$}\}$. The \emph{Jacobian algebra} $\cP(Q,W)$ is the quotient algebra $\widehat{KQ}/J(Q,W)$. In particular, if $Q$ is acyclic, then the potentials must be $0$ and the Jacobian algebra $\cP(Q,0)$ is just the path algebra $KQ$ of $Q$.

\begin{example}\label{exam:hereditary}
We consider finite dimensional hereditary algebras. It is well-known that any such an algebra $\Lambda$ is Morita equivalent to the path algebra $KQ$ of an acyclic quiver $Q$, that is, there is an equivalence of categories $\mod \Lambda \simeq \mod KQ$ (see e.g. \cite{ASS}), where $\mod \Lambda$ is the category of finitely generated left $\Lambda$-modules. Therefore, a finite dimensional hereditary algebra is Morita equivalent to some Jacobian algebra.
\end{example}

To define mutations of a QP, we need some preparation. We say that two potentials $W$ and $W'$ of $Q$ are \emph{cyclically equivalent} if $W-W'$ is contained in the closure of the span of all elements $\alpha_1\cdots\alpha_m-\alpha_2\cdots\alpha_m\alpha_1$ for oriented cycles $\alpha_1\cdots\alpha_m$. For a quiver $Q$, we denote by $Q_0$ the set of vertices of $Q$ and by $Q_1$ the set of arrows of $Q$. We say that two QPs $(Q,W)$ and $(Q',W')$ with $Q_0=Q'_0$ are \emph{right-equivalent} if there is an automorphism $\phi : \widehat{KQ} \simeq \widehat{KQ'}$ of the $K$-algebras such that $\phi$ is the identity on $Q_0$, and $\phi(W)$ and $W'$ are cyclically equivalent. 
We say that a QP $(Q,W)$ is
\begin{itemize}
\item \emph{reduced} if $W$ is a formal linear combination of oriented cycles of length at least three;
\item \emph{trivial} if $W$ is a formal linear combination of oriented cycles of length two and $J(Q,W)=\rad \widehat{KQ}$.
\end{itemize}

For two QPs $(Q,W)$ and $(Q',W')$ with $Q_0=Q'_0$, the direct sum $(Q,W) \oplus (Q',W')$ is a QP $(Q'',W+W')$ with $Q''_0=Q_0$ and $Q''_1=Q_1 \sqcup Q'_1$.

\begin{theorem}\cite[Theorem 4.6]{DWZ08}
For a QP $(Q,W)$, there are a reduced QP $(Q_{\rm red},W_{\rm red})$ and a trivial QP $(Q_{\rm triv},W_{\rm triv})$ such that $(Q,W)$ and $(Q_{\rm red},W_{\rm red}) \oplus (Q_{\rm triv},W_{\rm triv})$ are right-equivalent. They are unique up to right-equivalence.
\end{theorem}

We are ready to define the notion of mutations of a QP $(Q,W)$. We assume that $Q$ has no oriented cycles of length two incident to $k \in Q_0$. The \emph{mutation $\mu_k(Q,W)$ of $(Q,W)$ at $k$} is a reduced QP $(P_{\rm red},V_{\rm red})$, where the QP $(P,V)$ is defined as follows:
\begin{itemize}
\item $P$ is the quiver obtained from $Q$ by the following steps:
\begin{itemize}
\item[(1)] For any path $i \xrightarrow{a} k \xrightarrow{b} j$, add an arrow $i \xrightarrow{[ba]} j$.
\item[(2)] Replace any arrow $i \xrightarrow{a} k$ (resp., $k \xrightarrow{b} j$) with an arrow $i \xleftarrow{a^{\ast}} k$ (resp., $k \xleftarrow{b^{\ast}} j$).
\end{itemize}
\item By cyclically equivalence, we can assume that $W$ has no cycles starting and ending at $k$. Then $V=[W]+\Delta$, where
\begin{itemize}
\item[(1)] $[W]$ is obtained from $W$ by substituting $[ba]$ for any factor $ba$ with $i \xrightarrow{a} k \xrightarrow{b} j$;
\item[(2)] $\Delta=\sum_{(i \xrightarrow{a} k \xrightarrow{b} j) \in Q} [ba]a^{\ast}b^{\ast}$.
\end{itemize}
\end{itemize}

We say that a potential $W$ of $Q$ is \emph{non-degenerate} if every quiver obtained from $(Q,W)$ by any sequence of mutations has no oriented cycles of length two.

\begin{theorem}\cite[Corollary 7.4]{DWZ08}\label{thm:non-degenerate potential}
Let $Q$ be a quiver without loops nor oriented cycles of length two. If $K$ is uncountable, then there exists a non-degenerate potential of $Q$. In particular, if $Q$ is acyclic, then the potential $0$ is non-degenerate without assumption on $K$.
\end{theorem}

\begin{remark}
For a skew-symmetric matrix $B$, the diagram $\Gamma(B)$ is just a quiver without loops nor oriented cycles of length two. Thus $\Gamma(B)$ has a non-degenerate potential $W$. Then the quiver of $\mu_k(\Gamma(B),W)$ for $k \in \Gamma(B)_0$ is equal to $\Gamma(\mu_k B)$ \cite[Proposition 7.1]{DWZ08}.
\end{remark}

%%% $g$-vector fans on $2$-term silting theory
%%%
\subsection{$g$-vector fans on $2$-term silting theory}\label{subsec:g-vector fans on 2-term}

Let $\Lambda$ be a $K$-algebra which is one of the following:
\begin{enumerate}
\item a finite dimensional algebra;
\item a Jacobian algebra of a QP.
\end{enumerate}
We refer to \cite{AI,ASS,Kim} for the basic notations in the representation theory and $2$-term silting theory of $\Lambda$. We denote by $\proj\Lambda$ the category of finitely generated projective left $\Lambda$-modules and by $K^b(\proj\Lambda)$ the homotopy category of bounded complexes over $\proj\Lambda$ with suspension functor $[1]$. Then $K^b(\proj\Lambda)$ is a Krull-Schmidt category (see \cite[Lemma 2.17]{KY} or \cite[Corollary 4.6]{KM} for Jacobian algebras).

\begin{definition}
We say that an object $X \in K^b(\proj\Lambda)$ is
\begin{itemize}
\item \emph{presilting} if $\Hom_{K^b(\proj\Lambda)}(X,X[i])=0$ for all $i>0$;
\item \emph{silting} if it is presilting and $K^b(\proj\Lambda)$ is equal to the smallest triangulated subcategory of $K^b(\proj\Lambda)$, containing $X$, closed under direct summands.
\end{itemize}
We say that a presilting object in $K^b(\proj\Lambda)$ is \emph{$2$-term} if it is concentrated in degrees $-1$ and $0$.
\end{definition}

We denote by $\ipsilt\Lambda$ (resp., $\twosilt\Lambda$) the set of isomorphism classes of indecomposable $2$-term presilting (resp., basic $2$-term silting) objects in $K^b(\proj\Lambda)$.

The Grothendieck group of $K^b(\proj\Lambda)$, denoted by $K_0(\proj\Lambda)$, is a free abelian group whose basis is given by isomorphism classes of indecomposable projective $\Lambda$-modules. Thus this induces an isomorphism
\[
K_0(\proj\Lambda) \simeq \bZ^n,
\]
where $n$ is the number of indecomposable direct summands of $\Lambda$.

\begin{definition}
For an object $X \in K^b(\proj\Lambda)$, the \emph{$g$-vector} of $X$ is the image $[X] \in K_0(\proj\Lambda)$ of $X$.
\end{definition}

The $g$-vectors of $2$-term presilting objects in $K^b(\proj\Lambda)$ form a fan. It follows from \cite[Theorem 6.5]{DIJ} for finite dimensional algebras and from \cite[Proposition 3.1]{P11b} for Jacobian algebras (see also \cite{DF,H}).

\begin{theorem}\cite{DIJ,P11b}\label{thm:silting-gvector-fan}
There is a simplicial fan each of whose 
\begin{enumerate}
\item ray is spanned by the $g$-vector of an indecomposable $2$-term presilting object in $K^b(\proj\Lambda)$;
\item maximal cone is spanned by the $g$-vectors of indecomposable direct summands of a $2$-term silting object in $K^b(\proj\Lambda)$.
\end{enumerate}
\end{theorem}

\begin{definition}\label{def:silting-gvector-fan}
The fan in Theorem \ref{thm:silting-gvector-fan} is called the \emph{$g$-vector fan of $\Lambda$}, and it is denoted by $\cF(\Lambda)$.
\end{definition}

\begin{definition}
The algebra $\Lambda$ is called \emph{$g$-tame} if the $g$-vector fan $\cF(\Lambda)$ is dense.
\end{definition}

It is known that the following class is $g$-tame:
\begin{itemize}
\item $\tau$-tilting finite algebras \cite{DIJ};
\item Jacobian algebras associated with triangulated surfaces \cite{Y20};
\item finite dimensional tame algebras \cite{PY} (in particular, path algebras of extended Dynkin quivers \cite{H});
\item complete special biserial algebras \cite{AY};
\item complete preprojective algebras of extended Dynkin quivers \cite{KM};
\end{itemize}
Note that finite dimensional special biserial algebras and most of Jacobian algebras associated with triangulated surfaces are finite dimensional tame algebras (see \cite{WW} for the former and \cite{GLFS} for the latter). Theorem \ref{thm:g-tame hereditary} completely characterizes $g$-tame hereditary algebras.

On the other hand, on silting theory, mutations are defined as follows: Let $T=X \oplus U \in K^b(\proj\Lambda)$ be a silting object with $X$ indecomposable. If there exists a minimal left $(\add U)$-approximation $f$ of $X$, then we consider a triangle in $K^b(\proj\Lambda)$
\[
X \overset{f}{\longrightarrow} U' \longrightarrow Y \longrightarrow X[1].
\]
Then $Y \oplus U$ is a silting object, called the \emph{left mutation} of $T$ with respect to $X$. Similarly, if there exists a minimal right $(\add U)$-approximation $g$ of $X$, then we consider a triangle in $K^b(\proj\Lambda)$
\[
Z \longrightarrow U'' \overset{g}{\longrightarrow} X \longrightarrow Z[1].
\]
Then $Z \oplus U$ is a silting object, called the \emph{right mutation} of $T$ with respect to $X$. A \emph{mutation} of $T$ with respect to $X$ is a left or right mutation of $T$ with respect to $X$. If $\Lambda$ is finite dimensional, then it is clear that there are the above approximations $f$ and $g$ for any silting object $T=X \oplus U$. If $\Lambda$ is a Jacobian algebra, then it follows from \cite{KY} that there are the above approximations $f$ and $g$ for a silting object $T=X \oplus U$ obtained from $\Lambda$ by a sequence of mutations. In general, it is not known if the approximations exist. Moreover, it is not necessary that a mutation of a $2$-term silting object is $2$-term.

\begin{theorem}\label{thm:existmutations}
Let $T=X \oplus U \in \twosilt\Lambda$ with $X$ indecomposable.
\begin{enumerate}
\item \cite[Corollary 3.8(a)]{AIR} If $\Lambda$ is finite dimensional, then exactly one mutation of $T$ with respect to $X$ is $2$-term.
\item \cite[Theorem 2.18]{P11a}\cite[Proposition 3.1]{P11b} If $\Lambda$ is a Jacobian algebra and $T$ is a $2$-term silting object obtained from $\Lambda$ by a sequence of mutations, then exactly one mutation of $T$ with respect to $X$ is $2$-term.
\end{enumerate}
\end{theorem}

After this, we restrict mutations to ones between $2$-term silting objects. We denote by $\twosilt^{\circ}\Lambda$ the subset of $\twosilt\Lambda$ consisting of basic $2$-term silting objects obtained from $\Lambda$ by sequences of mutations, that is, every silting object appearing in the process is $2$-term. We denote by $\ipsilt^{\circ}\Lambda$ the subset of $\ipsilt\Lambda$ consisting of indecomposable $2$-term presilting objects which are direct summands of objects in $\twosilt^{\circ}\Lambda$.

%%% Additive categorification of cluster algebras
%%%
\subsection{Additive categorification of cluster algebras}\label{subsec:categorification}

In this subsection, we recall an additive categorification of skew-symmetric cluster algebras in terms of $2$-term silting objects.

Let $B$ be a skew-symmetric matrix. Then the diagram $\Gamma(B)$ is just a quiver without loops nor oriented cycles of length two. By Theorem \ref{thm:non-degenerate potential}, there exists a non-degenerate potential of $\Gamma(B)$ under the assumption that $K$ is uncountable (this assumption is not necessary if $B$ is acyclic). We fix such a potential $W(B)$ and denote the associated Jacobian algebra $\cP(B):=\cP(\Gamma(B),W(B))$. The following is due to Theorem \ref{thm:existmutations}(2) and the results in \cite{N} or \cite{P11b} (see \cite[Theorem 7.9]{Ke}). It also follows from \cite[Theorems 3.2 and 4.7]{AIR}, \cite[Theorem 6.3]{FK} and \cite[Corollary 3.5]{CIKLFP} via $\tau$-tilting theory for the case that $\cP(B)$ is finite dimensional.

%%%
\begin{theorem}[Additive categorification]\label{thm:categorification}
Let $B$ be a skew-symmetric matrix. Then there is a bijection
\[
\ipsilt^{\circ}\cP(B) \leftrightarrow \{\text{Cluster variables of $\cA(B)$}\}
\]
that preserves $g$-vectors. It induces a bijection
\[
\twosilt^{\circ}\cP(B) \leftrightarrow \{\text{Clusters of $\cA(B)$}\}
\]
that sends $\cP(B)$ to the initial cluster and commutes with mutations. In particular, $\cF(B)$ coincides with the subfan of $\cF(\cP(B))$ consisting of $g$-vectors of indecomposable $2$-term presilting objects in $\ipsilt^{\circ}\cP(B)$.
\end{theorem}

Finally, we get Theorem \ref{thm:g-tame hereditary}.

\begin{proof}[Proof of Theorem \ref{thm:g-tame hereditary}]
By Example \ref{exam:hereditary}, any finite dimensional hereditary algebra is Morita equivalent to a Jacobian algebra $J(Q,0)$ for some acyclic quiver $Q$. On the other hand, there is a skew-symmetric matrix $B$ whose diagram is $Q$, thus $\cP(B)=J(Q,0)$. Moreover, any basic $2$-term silting object can be obtained from the initial one for the acyclic case \cite[Proposition 3.5]{BMRRT}. The assertion follows from Theorems \ref{thm:g-dense acyclic} and \ref{thm:categorification}.
\end{proof}

%%%%%%%%%%%%%%%%%%%%%%%%%%%%%%%%%
% Acknowledgements
%%%%%%%%%%%%%%%%%%%%%%%%%%%%%%%%%
\medskip\noindent{\bf Acknowledgements}.
 The author would like to thank Osamu Iyama for his guidance and helpful comments. He also thanks Tomoki Nakanishi for useful discussion about rank $2$ scattering diagrams and Nathan Reading for helpful comments on Theorem \ref{thm:affine type dense}.

%%%%%%%%%%%%%%%%%%%%%%%%%%%%%%%%%
% References
%%%%%%%%%%%%%%%%%%%%%%%%%%%%%%%%%

\end{document}